\documentclass[10pt,longbibliography]{article}

\usepackage[utf8]{inputenc}

\usepackage[T1]{fontenc}

\usepackage[a4paper, margin=2.0cm]{geometry}

\usepackage{amsmath, amsthm, amsfonts, amssymb}
\usepackage{mathrsfs}           

\usepackage[colorlinks=true,linkcolor=blue,citecolor=blue,urlcolor=blue,breaklinks]{hyperref}

\usepackage{bbm}

\usepackage[numbers]{natbib}
\usepackage{url}
\usepackage{graphicx} 
\usepackage{amssymb}
\usepackage{bbm}
\usepackage{amsmath}
\usepackage{eucal}
\usepackage{tikz-cd}
\usepackage{amsthm}
\usepackage{stmaryrd}
\usepackage{blindtext} 
\usepackage{epigraph} 

\newtheorem{Theo}{Theorem}[section]
\newtheorem{Def}[Theo]{Definition}
\newtheorem{Ex}[Theo]{Example}
\newtheorem{Example}[Theo]{Example}
\newtheorem{Prop}[Theo]{Proposition}
\newtheorem{Lemma}[Theo]{Lemma}
\newtheorem{Corolary}[Theo]{Corollary}

\newcommand{\C}{\mathcal{C}}
\newcommand{\Sets}{\mathcal{S}\text{ets}}

\newcommand{\One}{\mathbbm{1}}
\newcommand{\Po}{\mathbb{P}}
\newcommand{\B}{\mathbb{B}}
\newcommand{\He}{\mathbb{H}}
\newcommand{\Qo}{\mathbb{Q}}
\newcommand{\D}{\mathcal{D}}

\newcommand{\N}{\omega}

\newcommand{\Pk}{\mathbb{P}_{\kappa}}

\newcommand{\hC}{\Hat{\C}}

\newcommand{\ydot}{\overset{.}{y}}
\newcommand{\zdot}{\overset{.}{z}}
\newcommand{\xdot}{\overset{.}{x}}
\newcommand{\wdot}{\overset{.}{w}}
\newcommand{\forces}{\Vdash^{*}}
\newcommand{\Forces}{\Vdash}
\newcommand{\lb}{\llbracket}
\newcommand{\rb}{\rrbracket}
\newcommand{\us}{\overset{\to}{u}}
\newcommand{\lp}{\overset{\leftarrow}{p}}
\newcommand{\x}{\overset{\to}{x}}

\usepackage{setspace}

\title{\emph{Relating forcing relations}}
\author{Michel Viana Smykalla \\ { \small TALTECH} \\
{\small michel.viana@taltech.ee} \and Hugo Luiz Mariano \\ {\small IME-USP} \\ {\small hugomar@ime.usp.br}}
\date{}

\begin{document}

\maketitle
\doublespacing
\begin{abstract}
Forcing was first introduced by Paul J. Cohen in his work on the independence of the Continuum Hypothesis. Other formulations of forcing appeared using Model Theory, Boolean-valued Models, and Topos Theory. There is a folkloric claim that these three approaches are equivalent, at least at the level of their mathematical content. In this work, we present some results not found in the literature toward establishing connections between these versions of forcing.

{\bf Keywords:} forcing, Boolean-valued models, sheaves, topos theory
\end{abstract}

\section*{Introduction}

To complete the proof of the independence of the Continuum Hypothesis ($CH$) from the axioms of Zermelo-Fraenkel Set Theory including the Axiom of Choice ($ZFC$), Paul J. Cohen developed a technique called forcing, which allows us to extend models of $ZFC$ and, in particular, to present an extension where the $CH$ fails, see \cite{Cohen63} and \cite{Cohen64}. For the first time in the history of Mathematics, one mathematical statement was proven to be undecidable inside the theory, and it started the era of independent results. Other formulations of forcing appeared almost at the same time in the 1960s, given the fact that the use of constructable sets could be replaced by constructing a Boolean-valued Model, done by the hands of Dana Scott and Solovay \cite{Scott67}. A decade later, formulations in Topos Theory appeared (Mitchell-Bénabou language, Kripke Joyal semantics, see \cite{MacMoe94} Chapter VI). Forcing ideas were also introduced in Model Theory in the 1970s by Shoenfield \cite{Shoen71}.

There is a well-known claim that these three approaches of forcing in Set-theory  are the same, at least at the level of their mathematical content:

\begin{center}
\emph{``Nevertheless, it is our clear understanding that the ultimate mathematical content of all these methods (generic sets, Boolean-valued models, and double-negation sheaves) is essentially the same. Indeed, a reading of the original paper by Paul Cohen clearly reveals the role there of double-negation. And sheafification has a wraith-like presence in Cohen's paper. 
Perhaps a full understanding makes use of all three approaches --- generic sets, sheaves, Boolean-valued Models!''} \\ Saunders Mac Lane and
Ieke Moerdijk \cite{MacMoe94}
    
\end{center}


However, to the best of our knowledge, there is no published paper containing a precise/complete description of these equivalences. The aim of this paper is to shed some light on this, building on the results 
presented in the first author's Master's Thesis, \cite{Smy24}.

{\bf Outline of the paper:} We begin by presenting the basic definitions in forcing in Section \ref{Forcingnotation}, to establish the notation. Section \ref{rfr} is devoted to present a comparison between forcing notions between posets that we can find in the literature, under the existence of a particular poset morphism called dense morphism. In Section \ref{forcingbooleanandposets}, we propose a point of view of where forcing with posets and forcing semantics in Boolean-valued models are the same. The key step here is to unify the definitions of name and forcing relation. Section \ref{forcingmodelsandsheaves} is devoted to compare  Boolean-valued models and topos of sheaves over complete Boolean algebras and to establish a generalization of a well-known result connecting notions of forcing using sheaf theory, thanks to a generalization of the so called Comparison Lemma in Topos Theory. We finish this work in Section \ref{futureworks}, showing directions of future research concerning categorical and semantical questions.

From now on, we assume that the definitions refer to sets inside a fixed countable transitive model $M$, pointing out when it is not the case.

\section{Basic forcing notation}\label{Forcingnotation}
Before we go to the comparison between the methods, first we introduce the basic definitions of forcing with posets and generic filters, and forcing through Boolean-valued models. For a complete presentation of each version, see \cite{Jech03} and \cite{Kunen11}, respectively. 

\begin{Def}\label{forcingposetdef}
	Let  $\Po$ be a non-empty set, $\One \in \Po$ and $\leq$ a relation on $\Po$. We say that the triple $(\Po, \leq, \One)$ is a \textbf{forcing poset} if the relation $\leq$ is a pre-order and for all $p \in \Po$, $p \leq \One$.
\end{Def}

\begin{Ex}
	Let $\N$ be the set of natural numbers with the reverse inequality $\succcurlyeq$, i.e., for all $x,y \in \N$
	$$x \succcurlyeq y \text{ iff } y \leq x.$$
	Then $(\N,\succcurlyeq,0)$ is a forcing poset.
\end{Ex}

\begin{Ex}\label{Cohenforcingdef}
	Fix a cardinal $\kappa > \aleph_{1}$ and consider the following set:
	$$\Pk= \{ f \in 2^{\omega \times \kappa}: |f| < \aleph_{0} \}.$$
	That is, $\Pk$ is the set of all finite functions from $ \omega \times \kappa$ to $2$. For all $f,g \in \Pk$, $f \leq g$ if $g \subseteq f$ as a function. With this order, $\Pk$ is a forcing poset, which we call by \textbf{Cohen forcing}.
\end{Ex}

\begin{Def}
	Let $\Po$ be a forcing poset and $D \subseteq \Po$. Then $D$ is \textbf{dense} in $\Po$ if for all $p \in \Po$, there exists $d \in D$ such that $d \leq p$. 
\end{Def}

\begin{Def}
	Let $\Po$ be a poset. A subset $D \subseteq \Po$ is \textbf{predense} if for all $p \in \Po$, there exists $d \in D$ such that $p$ and $d$ are compatible, that is, there exists $r \in \Po$ such that $r \leq p,d$.
\end{Def}

\begin{Def}
	Let $\Po$ be a forcing poset and $D \subseteq \Po$. Given $p \in \Po$, we say that $D$ is \textbf{dense below p} if for all $q \in \Po$ such that $q \leq p$, there exists $d \in D$ so that $d \leq q$.
\end{Def}

\begin{Def}
	Let $\Po$ be a forcing poset and $G \subseteq \Po$. Then $G$ is a \textbf{filter} on $\Po$ if
	
	\begin{enumerate}
		\item $\One \in G$.
		\item For all $p,q \in G$, there exists $r \in G$ such that $r \leq p,q.$
		\item For all $p, q \in \Po$, if $q \in G$ and $q \leq p$, then $p \in G.$
	\end{enumerate}
\end{Def}

\begin{Def}
	Let $\Po$ be a forcing poset, and $G$ be a filter on $\Po$. Then $G$ is $\Po-$\textbf{generic} (over $M$) if for all dense subset $ D \subseteq \Po$ such that $D \in M$, $G \cap D \neq \emptyset$.
	
\end{Def}

\begin{Def}\label{Pnamesdef}
	Let $\Po$ be a forcing poset. Then a set $\xdot$ is a ($\Po-)$\textbf{name} if $\xdot$ is a relation and the elements of $\xdot$ are of the form $(\ydot,p)$, where $\ydot$ is a $(\Po-)$name and $p \in \Po$. We denote by $M^{\Po}$ the class of all $\Po-$names. 
\end{Def}

\begin{Def}
	Let $\Po$ be a forcing poset and $G$ a filter on $\Po$. Given $\xdot$ a name, we define
	$$ \xdot_{G} = \{\ydot_{G}: \exists p \in G((\ydot,p) \in \xdot)\}.$$
	
\end{Def}

\begin{Def}
	Let $\Po$ be a forcing poset. Then 
	$$M[G] = \{\xdot_{G}~|~ \xdot \in M^{\Po}\}.$$
\end{Def}

\begin{Def}\label{forcingestreladef}
	Let $\Po$ be a forcing poset and $\xdot$ and $\ydot$ be $\Po-$names. We define recursively the \textbf{forcing relation} $\Vdash^*$ as follows: For all $p \in \Po$,
	\begin{enumerate}
		\item  p  $\Vdash^*$ $\xdot = \ydot$ iff for all $ \zdot \in dom(\xdot) \cup dom(\ydot)$ and for all $q \leq p(q \Vdash^* \zdot \in \xdot $ iff $q \Vdash^*  \zdot \in \ydot)$.  
		\item $p \Vdash^* \xdot \in \ydot $ iff the set $\{q \leq p: \exists (\zdot,r) \in \ydot$ which $q \leq r $ and $ q \Vdash^* \xdot = \zdot \}$ is dense below $p$.
		
		Let $\phi$ and $\psi$ be sentences of language of forcing. Then: 
		
		\item $p \forces \phi \land \psi$ iff $p \forces \phi$ and $p \forces \psi$.
		
		\item $p \forces \neg \phi$ iff  there is no $q \leq p$ such that $q \forces \phi$.
		
		\item $p \forces \phi \to \psi$ iff there is no $q \leq p$ such that $q \forces \phi \land \neg \psi$ 
		
		\item $p \forces \phi \lor \psi$ iff $\{q: q \forces \phi$ or $q \forces \psi \}$ is dense below p.
		\item $p \forces \phi \leftrightarrow \psi$ iff there are no $q \leq p$ such that $q \forces \phi \land \neg \psi$ and no $r \leq p$ such that $r \forces \neg \phi \land \psi$.
		\item $p \forces \forall x \phi(x)$ iff $p \forces \phi(\xdot)$, for all name $\xdot$.
		
		\item $p \forces \exists x\phi(x)$ iff the set $\{q \in \Po: \exists \xdot$ such that $q \forces \phi(\xdot)\}$ is dense below p.
		
	\end{enumerate}

\end{Def}

\begin{Lemma}\label{forcingwithposetslemma}
	Let $\Po$ be a forcing poset, $\phi$ be a formula of the language of forcing and $G$ be a generic filter on $\Po$ over $M$. Then
	\begin{enumerate}
		\item For all $p \in \Po$, if $p \in G$ and $p \forces \phi$, then $M[G] \models \phi$. 
		\item Suppose that $M[G] \models \phi$. Then there exists $p \in G$ so that $p \forces \phi$.
	\end{enumerate}
\end{Lemma}
\begin{proof}
	See \cite{Kunen11} Lemma IV.2.44.
\end{proof}

\begin{Lemma}\label{extensaoncrescemt}
	Let $\Po$ be a forcing poset and $G$ be a filter $\Po$-generic over M. Then:
	\begin{enumerate}
		\item $G \in M[G]$.
		\item $M[G]$ is a transitive model for $ZFC$.
		\item $M \subseteq M[G]$ and $M$ and $M[G]$ have the same ordinals. 
		\item If $N$ is a $ctm$ for $ZFC$, $G \in N$ and $M \subseteq N$, then $M[G] \subseteq N$.
	\end{enumerate}
\end{Lemma}
\begin{proof}
	See \cite{Kunen11} Lemmas IV.2.10, IV.2.12, IV.2.18, IV.2.19 and Theorem IV.2.27.
\end{proof}

Now, turning into Boolean-valued models, we have:

\begin{Def}\label{Booleanalgebradef}
	 Let $B$ be a nonempty set. A \textbf{Boolean algebra} is a $6-$tuple $(B,+,-,\cdot,0,1)$ such that: 
	\begin{enumerate}
		\item  $0$ and $1$  are elements of $B$.
		\item The symbols $+$ and $\cdot$ are functions from $B \times B$ to $B$ satisfying the following: For all $u,v$ and $w$ in $B$,
		\begin{enumerate}
			\item $u + v = v + u$ and $u \cdot v = v \cdot u$.
			
			\item $u + (v+w) = (u+v)+w$ and $u\cdot (v \cdot w) = (u \cdot v) \cdot w$.
			
			\item  $u\cdot(u+v) = u$ and $u + (u\cdot v) = u$.
			
			\item  $u \cdot(v + w) = u \cdot v + u \cdot w$ and $u+(v \cdot w) = (u+v)\cdot(u+w)$.
		\end{enumerate}
		\item The symbol $-$ is a function from $B$ to $B$ satisfying the following. For all $u$ in $B$,
		\begin{enumerate}
			\item $u + (-u) = 1$ and $u \cdot(-u) = 0$.
		\end{enumerate}
	\end{enumerate}
\end{Def}
For simplicity, we will abbreviate $(B,+,-, \cdot, 0,1)$ by $B$. For $u,v \in B$, 

\begin{Def}\label{Booleanvaluedmodeldef}
	Let $B$ be a complete Boolean algebra. A \textbf{Boolean-valued model} for set theory $\mathcal{U}$ consists of a transitive class $U$ equipped with two functions $$||-\in-||, || - = - ||: U \times U \to B,$$ such that:
	\begin{enumerate}
		\item $||x = x||= 1$.
		\item $||x=y|| = ||y=x||$.
		\item $||x = y|| \cdot ||y=z|| \leq ||x = z||$.
		\item $||x \in y|| \cdot ||v = x|| \cdot ||w=y|| \leq ||v \in w||$.
		
	\end{enumerate}
\end{Def}
With Boolean value for atomic formulas we can define a Boolean value of an arbitrary formula by induction on complexity. Let $\phi(\x)$ and $\psi(\x)$ be formulas with free variables $\x = (x_{1}.\dots, x_{n}).$ For $ \us =(u_{1},u_{2},\dots, u_{n})$ with $u_{i} \in U$, we have:
\begin{enumerate}
	
	\item $||\neg \phi(\us)|| = -||\phi(\us)||$.
	\item $||\phi(\us) \land \psi(\us)|| = ||\phi(\us)||\cdot ||\psi(\us)||$.
	\item $||\phi(\us) \lor \psi(\us)|| = ||\phi(\us)||+ ||\psi(\us)||$.
	\item $||\phi(\us) \rightarrow \psi(\us)|| = ||\neg\phi(\us)\lor \psi(\us)||$.
	\item $||\phi(\us) \leftrightarrow \psi(\us)|| = ||(\psi(\us)\to \phi(\us))\land (\psi(\us)\to \phi(\us))||$.
	
	\item $||\exists x \phi(x,\us)|| = \underset{x \in U}{\bigvee}||\phi(x,\us)||$.
	
	\item $||\forall x \phi(x,\us)|| = \underset{x \in U}{\bigwedge}||\phi(x,\us)||$.
\end{enumerate}

\begin{Def}
	
	Let $B$ be a complete Boolean algebra and denote by $V$ the  \emph{universe}. By recursion on the ordinals $\alpha \in Ord$, define the \textbf{Boolean-valued universe} $V^{B}$ as follows:
	
	\begin{enumerate}\label{BooleanvalueinV^B}
		\item $V_{0}^{B}$ = $\emptyset$.
		\item $V_{\alpha^{+}}^{B} = \{ f: f \text{ is a function with }im(f) \subseteq B\text{ and dom}(f) \subseteq V_{\alpha}^{B}\}.$
		\item $V_{\alpha}^{B} = \underset{\beta < \alpha}{\bigcup}V_{\beta}^{B}$, if $\alpha$ is a limit ordinal.
		\item $V^B = \underset{\alpha \in Ord}{\bigcup} V_{\alpha}^{B}.$ 
	\end{enumerate}
\end{Def}

One would say that $V^B$ is a kind of generalization of the traditional universe of sets $V$, and $V^B$ will be our transitive class. The Boolean value $||\phi||$ we will use with $V^B$ can be viewed as a generalization of the semantic consequence $V \models \phi$. The next definition will be done by recursion on the pair $(\rho(x),\rho(y))$, where $\rho(x)$ is the least ordinal such that $x \in V^{B}_{\rho(x) + 1}.$

\begin{Def}\label{BooleanValuesForV^B}
	Let $B$ be a complete Boolean algebra. For all $x,y \in V^B$, define:
	\begin{enumerate}
		\item $||x \in y|| = \underset{t \in dom(y)}{\bigvee}(||x=t||\cdot y(t)).$  
		\item $||x \subseteq y|| = \underset{t \in dom(x)}{\bigwedge}(-x(t) + ||t \in y||)$.
		\item $||x=y|| = ||x \subseteq y|| \cdot ||y \subseteq x||.$
	\end{enumerate}   
\end{Def}

\begin{Prop}
	
	Let $B$ be a complete Boolean algebra. Then $V^B$ equipped with the functions $||- \in - ||$ and $|| - = - ||$ (see Definition \ref{BooleanValuesForV^B}) forms a Boolean value model, which we will simply denote by $V^B$.
	
\end{Prop}
\begin{proof}
	See \cite{Jech03} Lemmas 14.15 and 14.16.
\end{proof}

\begin{Def}\label{M^Bdef}
	Let $M$ be a transitive model for $ZFC$ and $B \in M$ a complete Boolean algebra. We denote by $M^{B}$ the \textbf{Boolean-valued model constructed inside of} $M$. We call an element $\overset{.}{a} \in M^{B}$ by \textbf{name}, using an overhead dot notation.
\end{Def}

\begin{Def}\label{densemorphism}
	Let $(\Po,\leq_{\Po},\One_{\Po})$ and $(\Qo,\leq_{\Qo},\One_{\Qo})$ be two posets with top elements. A \textbf{dense morphism} from $\Qo$ to $\Po$ is a  function  $i: \Qo \to \Po$ satisfying the following properties:
	\begin{enumerate}
		\item $i(\One_{\Qo}) = \One_{\Po}$.
		\item If $q \leq_{\Qo} p$, then $i(q)\leq_{\Po} i(p)$, for all $p, q \in \Qo$.
		\item For all $p,q \in \Qo$, $q \perp_{\Qo} p$, if and only if $i(q) \perp_{\Po} i(p)$.
		\item $i(\Qo)$ is a dense subset of $\Po$.
	\end{enumerate}
\end{Def}

\begin{Lemma}\label{canonicaldensemorphism}
	Let $\Po$ be a poset. Then there exists a dense morphism
	$$i: \Po \longrightarrow \B \setminus \{0\}.$$
    for some complete Boolean algebra $\B$
\end{Lemma}
\begin{proof}
Consider the complete Boolean algebra $\B$ of the regular open subsets of the topological space  on $\Po$, endowed with the topology  generated by the subbasis $\{\overset{\leftarrow}{p} : p \in \Po\}$, where $\overset{\leftarrow}{p} = \{q \in \Po: q \leq p\}$. In particular $\leq = \subseteq$, $0 = \emptyset$, and $1 = \Po$.

	Define 
	\begin{align*}
		i : &\Po \longrightarrow \B \setminus \{\emptyset\} \\
		    &p \mapsto int(cl(\overset{\leftarrow}{p})),
	\end{align*}
	where $int(cl(\overset{\leftarrow}{p}))$ denotes the interior of the closure of $\overset{\leftarrow}{p}$. Note that, for all $p \in \Po$, $int(cl(\lp)) \in \B \setminus \{\emptyset\}$, because:
	
	First,
    $$p \in \lp \subseteq int(cl(\lp)).$$

    Moreover,
	$$int(cl(\lp)) \subseteq int(cl(int(\lp))).$$
	and
	$$int(cl(\lp)) \subseteq cl(\lp),$$
	then
	$$cl(int(cl(\lp))) \subseteq cl(\lp)$$
	and
	$$int(cl(int(cl(\lp)))) \subseteq int(cl(\lp)).$$
	It means that $int(cl(\lp))$ is a regular open set, because $int(cl(int(cl(\lp)))) = int(cl(\lp))$, so $int(cl(\lp)) \in \B \setminus \{\emptyset\}$.
	
	The next step is to show that $i$ satisfies the definition of a dense morphism.
	\begin{enumerate}
		\item Note that $\overset{\leftarrow}{\One_{\Po}} = \Po = \One_{\B}$. Therefore $i(\One_{\Po}) = \One_{\B}.$
		
		\item Let $p,q \in \Po$ such that $p \leq_{\Po} q$. Then $\overset{\leftarrow}{p} \subseteq \overset{\leftarrow}{q}$, and $$i(p) = int(cl(\overset{\leftarrow}{p})) \subseteq int(cl(\overset{\leftarrow}{q})) = i(q).$$
		Therefore $i(p) \leq_{\B} i(q).$
		\item Let $p,q \in \Po$ such that $p \perp_{\Po} q$. We want to prove that $i(p) \perp_{\B} i(q)$. Note that $p \perp_{\Po} q$ implies $\overset{\leftarrow}{p} \cap \overset{\leftarrow}{q} = \emptyset$. If there exists $ x \in \overset{\leftarrow}{p} \cap int(cl(\overset{\leftarrow}{q}))$,  $\overset{\leftarrow}{x} \cap \overset{\leftarrow}{q} \neq \emptyset$, then $\overset{\leftarrow}{p} \cap \overset{\leftarrow}{q} \neq \emptyset$. Therefore $\overset{\leftarrow}{p} \cap int(cl(\overset{\leftarrow}{q})) =  \emptyset$. In particular, we can prove by contradiction that it implies that $int(cl(\overset{\leftarrow}{p})) \cap int(cl(\overset{\leftarrow}{q})) = \emptyset$, and the desired result $i(p) \perp_{\B} i(q)$ follows. Now assume that $p \not \perp_{\Po} q$. Then there exists $r \in \Po$ so that $r \leq p,q$. We have that
		\begin{align*}
			&r \in \lp \subseteq int(cl(\lp)) = i(p), \\
			&r \in \overset{\leftarrow}{q} \subseteq int(cl(\overset{\leftarrow}{q})) = i(q).
		\end{align*}
		Therefore $\overset{\leftarrow}{r} \leq_{\B} i(p),i(q)$, which means that $i(p) \not \perp_{\B} i(q)$.
		
		\item Let $B \in \B \setminus \{\emptyset\}$. To finish this proof, we want to show that there exists $c \in i(\Po)$ such that $c \leq_{\B} B$. Fix a $b \in B$. In particular, $\overset{\leftarrow}{b} \subseteq B$. Therefore $$i(b) = int(cl(\overset{\leftarrow}{b})) \subseteq int(cl(B)) = B.$$
		Then, take $c = i(b)$.
		
	\end{enumerate}
\end{proof}

We call the morphism $i$ constructed above by \textbf{canonical dense morphism}.

\begin{Def}\label{forcingrelationbooleanalgebra}
	Let $\Po$ be a forcing poset, $\B$ be the complete Boolean algebra of regular open sets of $\Po$, and $i: \Po \to \B$ be the canonical dense morphism (see Definition \ref{canonicaldensemorphism}). Given a formula $\phi(\overset{.}{a}_{1},\overset{.}{a}_{2}, \dots, \overset{.}{a}_{n})$ of the language of forcing, where $\overset{.}{a}_{i} \in M^{\B}$, for all $p \in \Po$ ,
	$$p \Vdash \phi(\overset{.}{a}_{1},\overset{.}{a}_{2}, \dots, \overset{.}{a}_{n}) \text{ iff }  i(p) \leq ||\phi(\overset{.}{a}_{1},\overset{.}{a}_{2}, \dots, \overset{.}{a}_{n})||.$$
\end{Def}

\section{Relating forcing in posets}\label{rfr}

The first comparison of forcings we will explore is between forcings defined in two different posets, $\Po$ and $\Qo$. Of course, if $\Po \cong \Qo$, i.e., if there exists a bijection function $\Po \to \Qo$ which preserves and reflect their orders, then $\forces_{\Po}$ and $\forces_{\Qo}$ are the \emph{same}. We will see that a specific family of morphisms between forcing posets, the dense morphisms (see Definition \ref{densemorphism}), allows us to compare forcing relations and semantics, connecting their extensions via generic filters. 

To compare forcing relations in two posets, $\Po$ and $\Qo$, we need to find a way to compare their notion of names. The following definition presents a way to define a $\Po-$name from a $\Qo-$name, given a dense morphism.
\begin{Def}\label{i*definition}
	Let $\Po,\Qo$ be posets and $i: \Qo \to \Po$ a  morphism. Given a $\Qo-$name $\xdot$, define the $\Po$-name  $i_{*}(\xdot)$ by recursion $$i_{*}(\xdot) = \{(i_{*}(\ydot),  i(q)): (\ydot,q) \in \xdot \}.$$  
\end{Def}
In the definition above, it is unnecessary for $i: \Qo \to \Po$ to be a dense morphism. For example, we could remove the condition that $i(\Qo)$ is a dense subset of $\Po$. However, we choose to use a dense morphism because it will be necessary for the main theorems that connect forcing definitions.
\begin{Lemma}\label{completemorphismextensions}
	Let  $\Po$ and $\Qo$ be two posets and $i \in M$ a dense morphism $i: \Qo \to \Po$. Fix a $G \subseteq \Po$ a $\Po-generic$ filter over $M$ and define $H = i^{-1}(G)$. Then 
	\begin{enumerate}
		\item $H$ is $\Qo-$generic over $M$.
		\item $i_{*}(\xdot)$ is a $\Po-name$ and $$i_{*}(\xdot)_{G} = \xdot_{H},$$ for all $\xdot$ $\Qo-$name.
		\item $M[H] \subseteq M[G]$.
	\end{enumerate}
\end{Lemma}
\begin{proof}
	See \cite{Kunen11} Lemma IV.4.2 and Lemma IV.4.4.
\end{proof}

\begin{Lemma}
	Let $\Po$ be a forcing poset. Let $G$ and $H$ be two $\Po-generic$ filters over $M$ such that $G \subseteq H$, then $G = H$. 
\end{Lemma}
\begin{proof}
	First we show the inclusion $H \subseteq G$. Fix $h \in  H$. Define $$D =\{ p \in \Po: p \leq h \text{ or } p \perp h\}.$$ 
	
	Even though $H$ may not belong to $M$, its elements are in $M$. Using the fact that $M$ is a model for Set Theory, we can use set-theoretic axioms to construct $D$, and then $D \in M$. Given $q \in \Po$, there is two alternatives. If $q \perp h$, $q \in D$. Otherwise, there exists $r \in \Po$ such that $r \leq q,h$. Therefore $r \in D$. Which means that $D$ is dense, so $G \cap D \neq \emptyset$. Now take $s \in G \cap D$. By definition of $D$, as $G \subseteq H$, we have that $s \leq h$, then $h \in G$.
\end{proof}

\begin{Theo}\label{forcinganddensemorphism}
	Let $\Po$ and $\Qo$ be two forcing posets and $i: \Qo \to \Po$ be a dense morphism. Then the following assertions hold:
	\begin{enumerate}
		\item\label{item1} Let $H$ be a filter $\Qo-$generic over $M$. Define $G = \{p \in \Po: \exists q \in H(i(q) \leq p)\}$. Then $G$ is $\Po-$generic over $M$ and $i^{-1}(G) = H$.
		\item\label{item2} Let $G$ a filter $\Po-$generic over $M$ and take $H = i^{-1}(G)$. Then $H$ is $\Qo-$generic over $M$ and $G$ can be written as $G = \{p \in \Po: \exists q \in H(i(q) \leq p)\}$.
		
		\item\label{item3} In cases $(1)$ and $(2),$ we have the equality of the generic extension of $M$: $$M[G] = M[H].$$
		
		\item\label{item4} Let $\phi(x_{1},x_{2},\dots,x_{n})$ be a formula of the language of set theory and   $\xdot_{1},\xdot_{2},\dots,\xdot_{n} \in M^{\Qo}$ \\ $\Qo-$names. Then:
		$$q \forces_{\Qo} \phi(\xdot_{1},\xdot_{2},\dots,\xdot_{n}) \text{ if and only if } ~i(q) \forces_{\Po} \phi(i_{*}(\xdot_{1}),i_{*}(\xdot_{2}),\dots,i_{*}(\xdot_{n})).$$
	\end{enumerate}
\end{Theo}

\begin{proof}
	\begin{enumerate}
		\item First we prove that $G$ is a filter. Note that $i(\One_{\Qo}) = \One_{\Po} \in G$ and $G$ is closed upwards by definition. Now fix $p,q \in G$. Then there exists $p',q' \in H$ such that $i(q') \leq q$ and $i(p') \leq p$. $H$ is a filter, then there exists $r \in H$ such that $r \leq p'$ and $r \leq q'$. So we have $i(r) \in G$ such that $i(r) \leq p$ and $i(r) \leq q$. Now, let $D \subseteq \Po$ a nonempty dense open set of $\Po$. We need to show that $G \cap D \neq \emptyset$, and then $G$ will be $\Po-$generic. $i$ is a dense morphism, therefore $i^{-1}(D)$ is dense, which means that $H \cap i^{-1}(D) \neq \emptyset$. But if $p\cap i^{-1}(D) \neq \emptyset$, $i(p) \in G \cap D$. To conclude, note that by definition $H \subseteq i^{-1}(G)$, and both $H$ and $i^{-1}(G)$ are generic filters. Therefore $H  = i^{-1}(G)$.
		
		\item We know that $H = i^{-1}(G)$ is a $\Qo-$generic filter over $M$ (see Lemma \ref{completemorphismextensions}). By the item \ref{item1},  $$\{p \in \Po: \exists q \in H(i(q) \leq p)\}$$ is a $\Po-$generic filter. By definition, $G \subseteq \{p \in \Po: \exists q \in H(i(q) \leq p)\}$. Therefore $G = \{p \in \Po: \exists q \in H(i(q) \leq p)\}$.
		
		\item We need to show the inclusions. The first one we have by Lemma \ref{completemorphismextensions}, which shows that $M[H] \subseteq M[G]$. Note that $M \subseteq M[H]$, and $G \in M[H]$. Therefore, by { Lemma \ref{extensaoncrescemt}}, $M[G] \subseteq M[H]$.
		
		\item Suppose that $q \forces_{\Po} \phi(\xdot_{1},\xdot_{2},\dots,\xdot_{n}).$ Let $G$ be a filter on $\Po$. Define $H$ as in item \ref{item1}, so $i(q) \in G$ by item \ref{item2}. Using item \ref{item3}, $M[H] = M[G]$ and by hypothesis { (see Lemma \ref{forcingwithposetslemma})}, $$M[H] \models \phi(\xdot_{1},\xdot_{2}, \dots, \xdot_{n}).$$ However, $(\xdot_{j})_{H} = i_{*}(\xdot_{j})_{G}$, for all $j =1,2, \dots, n$. Therefore
		$$M[G] \models \phi(i_{*}(\xdot_{1}),i_{*}(\xdot_{2}),\dots,i_{*}(\xdot_{n})).$$
		In other words, $$i(q) \forces \phi(i_{*}(\xdot_{1}),i_{*}(\xdot_{2}),\dots, i_{*}(\xdot_{n})).$$ The reverse case is analogous. 
	\end{enumerate}
	
\end{proof}
In particular, the previous lemma is true if we consider $\Po$ to be the Boolean algebra of regular open sets of $\Qo$ (removing the bottom element $\emptyset$).

\section{Forcing as Boolean valuation}\label{forcingbooleanandposets}
 In Section \ref{Forcingnotation}, we introduced two notions of forcing relation: On the one hand, we defined $\forces$ using posets. On the other hand, we defined $\Forces$ through a complete Boolean algebra. Now, we will see that $\forces$ and $\Forces$ not only can produce the same independence results but are in a sense \emph{equal}. In other words, we will study a case where forcing with a poset will be the same thing as forcing with a Boolean-valued model. Here we have most of the original contributions of this work.




	
	


Fix a forcing poset $\Po$ and let $i: \Po \to \B \setminus \{\emptyset\}$ be the canonical dense morphism i.e., the one presented in the proof of  Lemma \ref{canonicaldensemorphism}.  By Lemma \ref{forcinganddensemorphism}, we have a correspondence between 
$$(\Po, \forces_{\Po})~\text{ and }~(\B\setminus \{\emptyset \},\forces_{\B \setminus \{\emptyset\}}).$$
The idea now is to relate ($\B\setminus \{\emptyset \},\Forces_{\B \setminus \{\emptyset\}})$ and $(\Po, \forces_{\Po})$.

\begin{Def}\label{functionalnamesdef}
	Let $\Po$ be a forcing poset and $\xdot$ be a $\Po-$name. Then $\xdot$ is a \textbf{functional} $\Po-$name if $\xdot$ is a function and for all $(\ydot, p) \in \xdot$, $\ydot$ is a functional $\Po$-name. We denote by $M^{\Po}_{fun}$ the class of all functional $\Po-$names. Similarly, given a complete Boolean algebra $\B$, an element $\xdot \in M^\B$ is a functional $\B-$name if it is a functional $(\B \setminus \{0\})-$name, and we denote by $~M^\B_{fun}$ the class of all functional $\B-$names.
\end{Def}

The difference between a name and a functional name is that the last is a functional relation. Definition \ref{functionalnamesdef} aims to approximate the notion of $\Po-$names using posets and $\B-$names in Boolean-valued models $M^\B$, once the latter is a function. The following lemma is part of the original contributions of this work.
\begin{Lemma}\label{retractdef}
	Let $\B$ be a complete Boolean algebra. Consider the forcing poset $\B \setminus \{0\} $. For the inclusion of functional $\B-$names  (see Definition \ref{functionalnamesdef}) into 
	$(\B \setminus \{0\})$-names {(see Definition \ref{Pnamesdef})}
	$$s: M^{\B}_{fun} \hookrightarrow M^{\B\setminus \{0\}},$$ 
	there exists a retraction $r: M^{\B \setminus \{0\}} \to M^{\B}_{fun}$ so that $r \circ s = id$.
\end{Lemma}
\begin{proof}
	By recursion, define:
	\begin{align*}
		r: &M^{\B\setminus \{0\}} \to M^{\B}_{fun} \\
		& \xdot \mapsto \{ (r(\ydot), \bigvee \{ p: (\zdot, p) \in \xdot \text{ and } r(\zdot) = r(\ydot)\}: \ydot \in dom(\xdot)) \}.
	\end{align*}
	For all $\xdot \in M^{\B\setminus \{0\}}, r(\xdot)$ is a name since $\B$ is complete, and it is a functional name because for all $\ydot_{1},\ydot_{2} \in \xdot$, if $r(\ydot_{1}) = r(\ydot_{2})$, then 
	$$\bigvee \{ p: (\zdot,p) \in \xdot \text{ and }r(\zdot) = r(\ydot_{1}) \}= \bigvee \{q: (\wdot,q) \in \xdot \text{ and } r(\wdot)= r(\ydot_{2})\}.$$
	It is easy to show that $r \circ s = id$ by induction on the $\in$ relation. Let $\xdot \in M^{\B}_{fun}$. Suppose that for all $\ydot \in dom(\xdot)$, $$r \circ s (\ydot) = r (\ydot) = \ydot.$$
	Then
	\begin{align*}
		r \circ s(\xdot)= r(\xdot) &= \{ (r(\ydot),\bigvee \{p_{\ydot}: (\ydot,p_{\ydot}) \in \xdot, \ydot \in dom(\xdot)  \})\}\\
		&= \{(\ydot,p_{\ydot}): \ydot \in dom(\xdot) \} \\
		&= \xdot.
	\end{align*}
In the equation above, we used that $\xdot$ is a functional name, then $\ydot$ is a functional name too and so the equalities hold.
\end{proof}

\begin{Lemma}
	Let $\B$ be a complete Boolean algebra. Consider the retraction $$r: M^{\B \setminus \{0\}} \to M^{\B}_{fun},$$ as in Lemma \ref{retractdef}. For all $\xdot \in M^{\B}$,
	$$\One \forces \xdot = r(\xdot).$$
\end{Lemma}
\begin{proof}
	See \cite{Kunen11} Exercise IV.4.23.
\end{proof}

Moreover, this retraction preserves forcing.

\begin{Lemma}\label{retractionpreservesforcing}
	Let $\B$ be a complete Boolean algebra. Fix a $p \in \B \setminus \{0\}$. Given a formula of the language of forcing $\phi(\xdot_{1},\xdot_{2}, \dots, \xdot_{n}),$ then
	$$ p\forces \phi(\xdot_{1},\xdot_{2}, \dots ,\xdot_{n}) \text{ if and only if } p \forces \phi(r(\xdot_{1}),r(\xdot_{2}),\dots,r(\xdot_{n})).$$
\end{Lemma}
\begin{proof}
	The proof follows easily by induction on the complexity. We will restrict ourselves to showing the case for the equality, as the other cases are similarly proven. Suppose $\phi(\xdot_{1},\xdot_{2},\dots,\xdot_{n })$ be $\xdot = \ydot$. Fix $G$ $\Po-$generic filter over $M$. By hypothesis, $M[G] \models \xdot_{G
	} = \ydot_{G}$. Let $r : M^{\Po} \to M^{\Po}_{fun}$ be the retraction as defined in the previous lemma. Then $M[G] \models \xdot = r(\xdot)$ and $M[G] \models \ydot = r(\ydot)$. Therefore, the next equality holds in $M[G]$
	$$r(\xdot) = \xdot = \ydot = r(\ydot).$$
	We conclude that $M[G] \models r(\xdot) = r(\ydot).$ Conversely, the argument is analogous. 
\end{proof}

In fact, there is a correspondence between a Boolean-valued model $M^\B$ {(see Definition \ref{M^Bdef})} and functional $\B-$names. To avoid ambiguity, we will denote the Boolean-valued model constructed inside of $M$ by $V^\B$ from now on. The following proposition composes the original contributions of this work.
\begin{Prop}\label{VBandfuncionalnames}
	Let $\B$ be a complete Boolean algebra. Then 
	$$V^{\B} \cong M^{\B}_{fun}.$$
\end{Prop}

\begin{proof}
	The idea is to construct the function with recursion on $\alpha \in Ord$, and prove by induction the bijection. For each ordinal $\alpha$, we want to define a function $$f_{\alpha}: V_{\alpha}^{\\B}\setminus \underset{\beta < \alpha}{\bigcup} V^{\B}_{\beta} \to M^{\B}_{fun},$$ paste them together and obtain a bijective function $$f : V^{\B} \to M^{\B}_{fun}.$$
	For $\alpha = 0$, $V_{0}^{\B} = \emptyset$. Then, define $f_{0} = \emptyset.$ In particular, $f_{0}$ is injective. Before we state the recursion assumption, we present how to define $f_{\alpha}$ if for all $\beta < \alpha$, $f_{\beta}$ is defined. Given $v \in V_{\alpha}^{\B} \setminus \underset{\beta < \alpha }{\bigcup}V^{\B}_{\beta}$, for each $u \in dom(v) $, let $\alpha_{u}$ be the first ordinal $\beta <\alpha$ in which $u \in V_{\beta}^{\B}$. Then, define: 
	$$ f_{\alpha}(v)  = \left\{\begin{array}{rll} \{(f_{\alpha_{u}}(u),v(u)): u \in dom(v)\},   & \hbox{if} & \alpha = \alpha_{u}. \\ f_{\alpha_{u}}(v),  & \hbox{if} & \alpha_{u} < \alpha .\end{array}\right. $$
	
	That being said, our recursion assumption will be: Suppose that for all $\beta < \alpha$, $$f_{\beta}: V_{\beta}^{\B} \setminus \underset{\gamma < \beta}{\bigcup} V_{\gamma} \to M^{\B}_{fun}$$ is injective  and for all $\beta_{1},\beta_{2} < \alpha$, if $\beta_{1} \neq \beta_{2}$ with $\alpha_{u} = \beta_{1}$ and $\alpha_{u'} = \beta_{2}$, $f_{\beta_{1}}(u) \neq f_{\beta_{2}}(u')$. Note that 
	
	\begin{enumerate}
		\item For all $v \in V^{\B}_{\alpha} \setminus \underset{\beta < \alpha}{\bigcup} V^{\B}_{\beta}$, $f_{\alpha}(v)$ is a functional $\B-$name. In fact, suppose that $(\xdot, p_{1}),(\xdot,p_{2}) \in f_{\alpha}(v)$. Then, there exists $u_{1},u_{2} \in dom(v)$ so that $$f_{\alpha_{u_{1}}}(u_{1}) = \xdot = f_{\alpha_{u_{2}}}(u_{2}).$$
		If $\alpha_{u_{1}} = \alpha_{u_{2}}$, then $f_{\alpha_{1}} = f_{\alpha_{2}}$. By the injective assumption, $u_{1} = u_{2}$ therefore $p_{1} = v(u_{1}) = v(u_{2}) = p_{2}.$ Otherwise, if $\alpha_{u_{1}} \neq \alpha_{u_{2}}$, then $$\xdot = f_{\alpha_{u_{1}}}(u_{1}) \neq f_{\alpha_{u_{2}}}(u_{2}) = \xdot,$$
		contradiction.
		\item $f_{\alpha}$ is injective, because given $v,v' \in V^{\B}_{\alpha} \setminus \underset{\beta < \alpha}{\bigcup} V^{\B}_{\beta}$ such that $v \neq v'$, $v$ and $v'$ differs as functions, therefore $f_{\alpha}(v) \neq f_{\alpha}(v')$.
	\end{enumerate}
	
	Then, define
	$$f = \underset{\alpha \in Ord}{\bigcup} f_{\alpha}: V^{\B} \to M^{\B}_{fun}.$$      
	
	Note that $f$ is injective because it is defined as a union of $2$-by-$2$ disjoint injective functions. It remains to show that $f$ is surjective.  Let $\xdot \in M^{\B}_{fun}$. By recursion on the well-founded relation $\in$, suppose that for all $\ydot \in dom(\xdot)$, exists a unique $\alpha_{\ydot}$ and a unique $u_{\ydot} \in V^{\B}_{\alpha_{\ydot}} \setminus \underset{\beta < \alpha_{\ydot}}{\bigcup} V^{\B}_{\beta}$ so that $f_{\alpha_{\ydot}}(u_{\ydot}) = \ydot$. Then, define 
	$$v_{\xdot} = \{(u_{\ydot}, p_{\ydot}): f_{\alpha_{\ydot}}(u_{\ydot}) = \ydot \text{ and }(\ydot,p_{\ydot}) \in \xdot\}.$$
	Then, we have $$f(v_{\xdot}) = \{(f_{\alpha_{\ydot}}(u_{\ydot}), v_{\xdot}(u_{\ydot})): u_{\ydot} \in v_{\xdot}\} = \xdot.$$
\end{proof}

Remember the definition of $i^{*}$ (see Definition \ref{i*definition}). Let $g: M^{\B}_{fun} \to V^{\B}$ be the inclusion function. With Proposition \ref{VBandfuncionalnames}, we have then completed the following diagram:

\begin{center}
	\begin{tikzcd}
		M^{\Po} \arrow[rr, "i_{*}"]                               &  & M^{\B \setminus \{\emptyset\}} \arrow[dd, "r", shift left]                                          \\
		&  &                                                                             \\
		M^{\Po}_{fun} \arrow[uu, hook] \arrow[rr, "i_{*}"] &  & M^{\B}_{fun} \arrow[dd, "g", shift left] \arrow[uu, "s", shift left] \\
		&  &                                                                             \\
		&  & V^{\B} \arrow[uu, "f", shift left]                                         
	\end{tikzcd}  
	
\end{center}

Using the definition of forcing with posets $\forces$, it is possible to assign Boolean-values to formulas. From this perspective, to some extent, one would say that assuming $\forces$, the definition of $\Forces$ becomes a theorem.

\begin{Def}\label{quadradinhokunen}
	 Let $B$ be a complete Boolean algebra. Consider the forcing poset $B \setminus \{0\}$. Given a sentence $\phi$ of forcing language, define
	$$ \lb \phi \rb =\bigvee \{b \in B \setminus \{0\} :  b \forces \phi\}.$$
\end{Def}

In particular, $\lb \phi \rb$ is the greatest element of $B \setminus \{0\}$ that forces $\phi$.

\begin{Lemma}\label{Booleanvalueforces}
	Let $B$ be a complete Boolean algebra. Let $\phi$ be a sentence of forcing language. For all $b \in B \setminus \{0\}$. Then
	$$b \forces \phi \text{ iff }~ b \leq \lb \phi \rb.$$
\end{Lemma}
\begin{proof}
	See \cite{Kunen11} Lemma IV.4.19.
\end{proof}

As the reader may have noticed, the notation $\lb \phi \rb$ is similar to the Boolean-value of a formula $|| \phi ||$ {(see Definition \ref{Booleanvaluedmodeldef})}. For the atomic cases, the similarity is even stronger.
\begin{Lemma}\label{quadradinhodokunenparaatomicas}
	Let $B$ a complete Boolean algebra. Given $\xdot, \ydot \in M^{B\setminus \{0\}}$,
	\begin{enumerate}
		\item $\lb \xdot \in \ydot \rb  = \underset{t \in dom(\ydot)}{\bigvee} ( \lb t = \xdot \rb \cdot \ydot(t)).$
		\item $\lb \xdot = \ydot \rb = \underset{t \in dom(\xdot) \cup dom(\ydot)}{\bigwedge} (\lb t \in \xdot \rb \iff \lb t \in \ydot \rb).$
	\end{enumerate}
\end{Lemma}

\begin{proof}
	See \cite{Kunen11} Exercise IV.4.24.
\end{proof}
The next step is to show when we can replace names for functional names, without any loss.

\begin{Def}
	We call by \textbf{forcing language with functional  $\Po$-names} the usual forcing language with posets $\Po$ in which we add only the functional names as constants.
\end{Def}
From Definition \ref{quadradinhokunen}, it is possible to derive the exact form for $\lb \phi \rb$ more complex formulas. Again, we emphasize the similarity with $||\phi||$ in {Definition \ref{Booleanvaluedmodeldef}}.

\begin{Lemma}\label{quadradinhokunenproperty}
	Let $B$ be a complete Boolean algebra. Let $\phi$ and $\psi$ be sentences of the forcing language with functional $(B \setminus \{0\})-$names. Then
	\begin{enumerate}
		\item $\lb \neg \phi \rb = - \lb \phi \rb.$
		
		\item $\lb \phi \land \psi \rb = \lb \phi \rb \cdot \lb \psi \rb.$
		
		\item $\lb \phi \lor \psi \rb = \lb \phi \rb + \lb \psi \rb.$
		
		\item $\lb \phi \to \psi \rb = \lb \neg \phi \lor \psi \rb$.
		
		\item $\lb \phi \leftrightarrow \psi \rb = \lb(\phi \to \psi)\land(\psi \to \phi) \rb$.
		
		\item $\lb \forall x \phi(x) \rb = \bigwedge \{\lb \phi(\xdot)\rb: \xdot \in M^{B}_{fun}\}.$
		
		\item $\lb \exists x \phi(x) \rb = \bigvee \{\lb \phi(\xdot) \rb: \xdot \in M^{B}_{fun} \}$.
	\end{enumerate}
\begin{proof}
	See \cite{Kunen11} Exercise IV.4.20 and Exercise IV.4.21.
\end{proof}
\end{Lemma}

Based on the Lemma \ref{quadradinhokunen} and Lemma \ref{quadradinhokunenproperty} , $|| - ||$ and $ \lb - \rb$ should share something in common. We will see that this is actually the case. Remember that $V^B$ can be viewed as a generalization of the universe of Set Theory, and its elements are not just functions but characteristic functions. Then for all $u \in V^B$, $u(x) = 0$ if $x \not \in dom(u)$. This convention appears in the textbooks, like in  \cite{Bell05} when the author proves that $V^B$ satisfies the axioms of Zermelo-Fränkel Set Theory. 
\begin{Lemma}\label{pricipiodomaximo}
	Let $B$ be a complete Boolean algebra. For all $u \in V^B$
	
	$$\underset{t \in V^B}{\bigvee} u(t) = \underset{t \in V^B}{\bigvee} ||t \in u ||.$$
\end{Lemma}
\begin{proof}
	By definition, for all $t \in dom(u)$, $u(t) \leq || t \in u||$. Therefore
	$$\underset{t \in V^B}{\bigvee} u(t) \leq \underset{t \in V^B}{\bigvee} || t \in u ||.$$
	On the other hand, $|| t \in u|| = \underset{s \in dom(u)}{\bigvee} ||t = s||\cdot u(s)$. Besides that, $v \leq 1$ for all $v \in V^B$. Thus
	\begin{align*}
		\underset{t \in V^B}{\bigvee} || t \in u || &\leq \underset{t \in V^B}{\bigvee}\left(\underset{s \in dom(u)}{\bigvee} ||t= s|| \cdot u(s)\right) \\
		&\leq \underset{t \in V^B}{\bigvee}\left(\underset{s \in dom(u)}{\bigvee} 1 \cdot u(s)\right) \\
		& = \underset{t \in V^B}{\bigvee} u(t).
	\end{align*}
\end{proof}

Given a complete Boolean algebra $B$, every functional $B$-name is in particular an element of $V^B$.
\begin{Lemma}\label{Equalityofboolanvalues}
	Let $B$ be a complete Boolean algebra. For all $\xdot, \ydot \in M^{B}_{fun}$
	\begin{enumerate}
		\item $\lb \xdot \in \ydot \rb = || \xdot \in \ydot ||$.
		\item $\lb \xdot = \ydot \rb = || \xdot = \ydot ||$.
	\end{enumerate}
\end{Lemma}

\begin{proof}
	First, item $1.$ follows if we prove that the item $2.$ holds. Therefore, we only need to prove $2.$, supposing that $1.$ is true. We will use the characterization of $\lb - \rb$ for atomic formulas (see Lemma \ref{quadradinhodokunenparaatomicas}). Remember the definition of operator $\implies$ in a Boolean algebra (see the discussion after {Definition \ref{Booleanalgebradef}}). Note that:
	
	

	\begin{align*}
		- \underset{t \in dom(\xdot)\cup dom(\ydot)}{\bigwedge}( \lb t \in \xdot \rb \implies \lb t \in \ydot \rb) &= \underset{t \in dom(\xdot)\cup dom(\ydot)}{\bigvee} - ( - \lb t \in \xdot \rb + \lb t \in \ydot \rb) \\
		& =\underset{t \in dom(\xdot)\cup dom(\ydot)}{\bigvee} (\lb t \in \xdot \rb \cdot (-\lb t \in \ydot \rb)) \\
		& = \underset{t \in dom(\xdot)\cup dom(\ydot)}{\bigvee} \lb t \in \xdot \rb \cdot \underset{t \in dom(\xdot)\cup dom(\ydot)}{\bigvee}(-\lb t \in \ydot \rb) \\
		& \overset{**}{=} \underset{t \in dom(\xdot)\cup dom(\ydot)}{\bigvee} \xdot(t) \cdot \underset{t \in dom(\xdot)\cup dom(\ydot)}{\bigvee} (- \lb t \in \ydot \rb) \\ 
		&= \underset{t \in dom(\xdot)\cup dom(\ydot)}{\bigvee}(\xdot(t) \cdot (- || t \in \ydot ||)) \\
		&= - \underset{t \in dom(\xdot)\cup dom(\ydot)}{\bigwedge}(\xdot(t) \implies || t \in \ydot||).
	\end{align*}
	
	In (**), we used Lemma \ref{pricipiodomaximo}. We conclude that 
	\begin{align*}
		\underset{t \in dom(\xdot)\cup dom(\ydot)}{\bigwedge} (\lb t \in \xdot \rb \implies \lb t \in \ydot \rb)
		&= \underset{t \in dom(\xdot) \cup dom(\ydot)} {\bigwedge} \xdot(t) \implies ||t \in \ydot||. \\
		& = \underset{t \in dom(\xdot)}{\bigwedge} \xdot(t) \implies || t \in \ydot ||.
	\end{align*}
	Therefore,
	\begin{align*}
		||\xdot = \ydot|| &= (\underset{t \in dom(\xdot)}{\bigwedge} \xdot(t) \implies || t \in \ydot||) \cdot (\underset{t' \in dom(\ydot)}{\bigwedge} \ydot(t') \implies ||t' \in \ydot||) \\
		&= (\underset{t \in dom(\xdot) \cup dom(\ydot)}{\bigwedge}\lb t \in \xdot \rb \implies \lb t \in \ydot \rb)\cdot(\underset{t' \in dom(\xdot)\cup dom(\ydot)}{\bigwedge} \lb t' \in \ydot \rb \implies \lb t' \in \xdot \rb) \\
		&= \underset{t \in dom(\xdot) \cup dom(\ydot)}{\bigwedge}(\lb t \in \xdot \rb \implies \lb t \in \ydot \rb) \cdot (\lb t \in \ydot \rb \implies \lb t \in \xdot \rb). \\
		&= \lb \xdot = \ydot \rb.
	\end{align*}
\end{proof}

\begin{Corolary}\label{osquadradinhossaoiguais}
	Let $B$ a complete Boolean algebra. For all formulas $\phi$ of the forcing language with $(B \setminus \{0\})-$names,
	$$\lb \phi \rb = ||\phi||.$$
\end{Corolary}
\begin{proof}
	The proof is by induction on the complexity of $\phi$. Lemma \ref{Equalityofboolanvalues} deals with the atomic cases, and Lemma \ref{quadradinhokunenproperty} deals with all the others.
\end{proof}

Now we present our formalization of the statement at the beginning of this section, saying that $\forces$
and $\Forces$ would be viewed as \emph{equal}. Remember the definition of the canonical morphism $i: \Po \to \B \setminus \{0\}$ {(see the comments  after Lemma \ref{canonicaldensemorphism})}.

\begin{Theo}\label{HugoMicheltheorem1}
	Let $\Po$ be a forcing poset and denote by $i: \Po \to \B$ the canonical dense morphism. Let $r: M^{\B \setminus \{0\}} \to M^{\B}_{fun}$ be the retraction as in Lemma \ref{retractdef}. Let $\phi(\xdot_{1},\xdot_{2},\dots, \xdot_{n})$ be a formula of the forcing language, where $\xdot_{1},\xdot_{2},\dots,\xdot_{n}$ are $\Po-$names. For all $p \in \Po$,
\end{Theo} 

\begin{align*}	
	p \forces \phi(\xdot_{1},\xdot_{2},\dots,\xdot_{n}) \text{\emph{ iff }} p \Forces \phi(r(i_{*}(\xdot_{1})),r(i_{*}(\xdot_{2})),\dots,r(i_{*}(\xdot_{n}))).
\end{align*}

 \begin{proof}
 	By Theorem \ref{forcinganddensemorphism},
 	$$p \forces \phi(\xdot_{1},\xdot_{2},\dots,\xdot_{n}) \text{\emph{ iff }} i(p) \forces \phi(i_{*}(\xdot_{1}),i_{*}(\xdot_{2}),\dots,i_{*}(\xdot_{n})).$$
 	By Lemma \ref{retractionpreservesforcing},
 	$$i(p) \forces \phi(i_{*}(\xdot_{1}),i_{*}(\xdot_{2}),\dots,i_{*}(\xdot_{n}))\text{\emph{ iff }} i(p) \forces \phi(r(i_{*}(\xdot_{1})),r(i_{*}(\xdot_{2})),\dots,r(i_{*}(\xdot_{n}))).$$
 	By Lemma \ref{Booleanvalueforces},
 	$$i(p) \forces \phi(r(i_{*}(\xdot_{1})),r(i_{*}(\xdot_{2})),\dots,r(i_{*}(\xdot_{n})))\text{\emph{ iff }} i(p) \leq \lb\phi(r(i_{*}(\xdot_{1})),r(i_{*}(\xdot_{2})),\dots,r(i_{*}(\xdot_{n}))) \rb.$$
 	
 	By Corollary \ref{osquadradinhossaoiguais},
 	$$i(p) \leq \lb\phi(r(i_{*}(\xdot_{1})),r(i_{*}(\xdot_{2})),\dots,r(i_{*}(\xdot_{n}))) \rb\text{\emph{ iff }} i(p) \leq ||\phi(r(i_{*}(\xdot_{1})),r(i_{*}(\xdot_{2})),\dots,r(i_{*}(\xdot_{n})))||.$$
 	Finally, by {Definition \ref{forcingrelationbooleanalgebra}},
 	$$i(p) \leq ||\phi(r(i_{*}(\xdot_{1})),r(i_{*}(\xdot_{2})),\dots,r(i_{*}(\xdot_{n})))||\text{\emph{ iff }} p \Forces \phi(r(i_{*}(\xdot_{1})),r(i_{*}(\xdot_{2})),\dots,r(i_{*}(\xdot_{n}))).$$
 \end{proof}

\section{Forcing, Boolean-valued models and sheaves over Boolean algebras}\label{forcingmodelsandsheaves}

Let $\B$ be a complete Boolean algebra. It is well known (see, for instance, \cite{Bell05} Appendix) that from the Boolean-valued model $V^\B$ can be extracted a category  $Set^{(\B)}$, "by taking quotients". In more detail:
\begin{itemize}
\item An object of  $Set^{(\B)}$ is a class of equivalence $[x]$, where $x \in V^\B$ and $[x] = [x']$ iff $||x = x'|| = 1_{\B}$;
\item An arrow $[f] : [x] \to [y]$ in $Set^{(\B)}$ is a class of equivalences $[f]$, where $f \in V^\B$,  $[f] = [f']$ iff $||f = f'|| = 1_{\B}$, and $|| f$ is a function with domain $x$ and range contained in $y|| = 1_\B$. 
\end{itemize}

Moreover, in Appendix  of \cite{Bell05}, it is sketched an equivalence of categories between this category obtained by quotients on the Boolean-valued universe $V^\B$ and the category of sheaves of sets on $\B$ obtained from the natural notion of covering given by suprema: $$Set^{(\B)} \simeq Sh_{\bigvee(\B)}$$.


Taking into account the above described  scenario, it is possible to establish a connection between the sheaf theoretic version of forcing and Boolean-valued models by showing an equivalence of categories. In this section, we will discuss this result as well as a generalization of it.

We will assume basic knowledge in category theory, introducing the concepts of Grothendieck topology and Grothendieck topos. Given a category $\C$, $Obj(\C)$ represents the (class of) objects of $\C$. As a reference for this introduction to Topos Theory, see \cite{MacMoe94}.

\begin{Def}\label{Sievedef}
	Let $\C$ be a category and $C \in \text{Obj}(\C)$. A \textbf{sieve} on $C$ is a family $S$
	of arrows in $\C$ all with codomain $C$, such that if $f: A \to C$ belongs to $S$ and $g: B \to A$ is any arrow in $\C$ with codomain $A$, then $f \circ g \in S$.
	
\end{Def}

If $\C$ is a locally small category, a sieve on an object $C$ of $\C$ will be a subobject $S$ of $y(C) = Hom_{\C}( - ,C)$, i.e., there exists a monomorphism $S \rightarrowtail y(C)$ in $\hC$ satisfying the universal property of subobjects of $y(C)$. Moreover, if $S$ is a sieve on $C$ and $h: D \to C$ is any arrow in $\C$ with codomain $C$, then 
$$h^{*}(S) = \{g\text{ }|\text{  codom}(g) = D \text{ and } h \circ g \in S\}$$
is a sieve on $D$.

\begin{Def}
	Let $\C$ be a small category. A \textbf{Grothendieck topology} on $\C$ is a function $J$ which associates to each object $C$ of $\C$ a family $J(C)$ of sieves on $C$ satisfying the following properties.
	\begin{enumerate}
		\item $t_{C} = \{ f\text{ } | \text{  codom}(f) = C\} \in J(C)$. We call $t_{C}$ the \textbf{maximal sieve}.
		
		\item If $S \in J(C)$, then for any arrow $h: D \to C$ in $\C$, $h^{*}(S) \in J(D).$ This property is known as the \textbf{stability axiom}.
		
		\item If $S \in J(C)$ and $R$ is a sieve on $C$ such that for any $h: D \to C \in S$, $h^{*}(R) \in J(D)$, then $R \in J(C)$. Some books call this condition by the \textbf{transitivity axiom}. 
	\end{enumerate}
\end{Def}

\begin{Def}\label{Sitedef}
We call by \textbf{site} a pair $(\C,J)$, composed by a small category $\C$ and a Grothendieck topology $J$ on $\C$. If $S \in J(C)$, we say that $S$ \textbf{covers} $C$. 
\end{Def}

Using a Grothendieck topology we will construct the category where the $CH$ is not satisfied. The next example is the topology that in fact will provide the Cohen topos.

\begin{Example}
	Let $\Po$ be a poset. Note that $\Po$ is a category, where an arrow $p \to q$ in $~\Po$ means that $ p\leq q$. Given $p \in \Po$, consider the set $\overset{\leftarrow}{p} = \{q \in \Po\text{ } | \text{ }q \leq p\}$. A subset $D \subseteq \overset{\leftarrow}{p}$ is \textbf{dense below $p$} if for every $r \leq p$, there exists $q \in D$ such that $q \leq r$. The dense sieves form a Grothendieck topology $J$ on $\Po$ by
	
	$$J(p) = \{D : D \text{ is a sieve on } p \text{ and dense below p}\}. $$
	We call this Grothendieck topology by \textbf{dense topology} or \textbf{double-negation topology}, denoted by $\neg\neg$-topology or just $\neg\neg$.  
\end{Example}
 A generalization of Boolean algebras is Heyting algebras, where we remove the \emph{excluded middle}, that is, $-(-u) = u$ is not true in  a Heyting algebra.
\begin{Example}
	Let $H$ be a complete Heyting algebra. We can see $H$  as a category in the same way as posets (an arrow $h \to k$ in $H$ means that $h \leq k$, for $h,k \in H$). The \textbf{sup topology} on $H$ is a Grothendieck topology $J$ such that for all $h \in H$,
	
	$$J(h) = \{S: \bigvee  S = h\}.$$
We usually denote the sup topology by $\bigvee$.
\end{Example}

\begin{Def}
	Let $\C$ be a category with pullbacks. A \textbf{basis} for a Grothendieck topology on $\C$ is a function $B$ which associates an object $C \in \text{Obj}(\C)$ to a collection $B(C)$ of families of arrows in $\C$ with codomain $C$ satisfying the next conditions:
	
	\begin{enumerate}
		\item If $f: C' \to C$ is an iso, then $~\{f:C' \to C\} \in B(C)$.
		
		\item If $~\{f_{i}: C_{i} \to C ~ |\text{ for all }i \in I\} \in B(C)$ and $g: D
		 \to C$ is any morphism of $\C$ with codomain $C$, then the family of projections $\{ \pi_{2}: C_{i} \times D \to D~|\text{ for all } i \in I\}$ belongs to $B(D)$.
		 
		\item If $~\{f_{i}: C_{i} \to C ~ |\text{ for all }i \in I\} \in B(C)$ and for each $i \in I$, there exists a family $\{g_{ij}: D_{ij} \to C_{i}~|~ j \in J_{i}\} \in B(C_{i})$, then $\{f_{i} \circ g_{ij}: D_{ij} \to C ~|\text{ for all }i \in I, j \in J_{i}\} \in B(C)$.
	\end{enumerate}
\end{Def}
\begin{Def}
	Let $\C$ be a category and $A \in \text{Obj}(\C)$. Given a collection of morphism $B = \{g_{i}:A_{i} \to A\}_{ i \in I }$ in $\C$, the sieve (see Definition \ref{Sievedef}) \textbf{generated} by $B$, is the sieve $$<B> = \{f \circ g_{j}: f: C \to A_{j} \text{ and } g_{j} \in B\}.$$ 
\end{Def}

Given a basis $B$ on $\C$, we can obtain a Grothendieck topology $J$ defining for each $C \in \text{Obj}(C)$, for every sieve $S$ on $C$,
$$S \in J(C) \text{ if and only if there exists } R \in B(C)\text{ such that }<R> \subseteq S. $$

The objects of the category that we are looking for (Grothendieck topos) are sheaves. To talk about them, we need more definitions, starting with the notion of presheaves. 

\begin{Def}\label{presheafdef}
	Let $\C$ be a category. A \textbf{presheaf} $F$ on $\C$ is a functor
	$$F: \C^{op} \to \Sets.$$
	
\end{Def}

\begin{Example}\label{Representablepresheaf}
One example of a family of presheaves are the functors of the form $y(C)=Hom_{\C}( - ,C)$ for a given object $C$ of $\C$. We call them by \textbf{representable presheaves}.
\end{Example}

\begin{Def}\label{categoryofpresheavesdef}
	
	Let $\C$ be a category. The \textbf{category of presheaves on} $\C$ is the category ${\Sets}^{\C^{op}}$ of functors $F: \C^{op} \to \Sets$ from the opposite category of $\C$ to $\Sets$ and natural transformations between them. In this case, we represent this category of presheaves over $\C$ using the notation $\Hat{\C}$. 
	
\end{Def}

Now, fix a site $(\C,J)$ and a presheaf $P$ on $\C$. Given an arrow $g: E \to D$ in $\C$, we have that $P(g) : P(D) \to P(E)$. For all $x \in P(D)$, $x \cdot g$ stands for $P(g)(x)$.  If $S$ is sieve and covers an object $C$ of $\C$, a \textbf{matching family} for $S$ of elements of $P$ is a function which associates each element $f: D \to C$ of $S$, to an element $x_{f} \in P(D)$ satisfying $$P(g)(x_{f}) = x_{f} \cdot g = x_{f\circ g},$$ for all morphisms $g: E \to D$ of $\C$. An \textbf{amalgamation} of such a matching family is a single element $x \in P(C)$ such that $$x \cdot f = x_{f} \text{ for all } f \in S.$$

\begin{Def}\label{Sheafdef}
	Let $(\C,J)$ be a site and $P$ a presheaf over $\C$. Then $P$ is a \textbf{sheaf for} $J$ if for every matching family  of elements of $P$ for any cover of any object of $\C$ there exists a unique amalgamation. In this case, we also say that $P$ is a sheaf on the site $(\C,J$).
\end{Def}

Then, sheaves on a site $(\C,J)$ form a category $Sh(\C,J)$, where the objects are the sheaves and the arrows, natural transformations between them. In this case, $Sh(\C, J)$ is a full subcategory of $\Sets^{\C^{op}}$, then we have the inclusion functor

\begin{center}
	\begin{tikzcd}
		{Sh(\C,J)} \arrow[r, "Id_{Sh(\C,J)}", hook] & \Sets^{\C^{op}}.
	\end{tikzcd}
\end{center}

\begin{Def}
	A \textbf{Grothendieck topos} is a category which is equivalent to the category $Sh(\C,J)$ of sheaves on some site $(\C,J)$.
\end{Def}

\begin{Ex}
	The following categories are Grothendieck toposes and we will use them later:
	\begin{enumerate}
		\item The category of sets, $\Sets$.
		\item The category of presheaves $~\widehat{\C} = \Sets^{\C^{op}}$, where $\C$ is a small category.
	\end{enumerate}
\end{Ex}

\begin{Def}\label{separativeposet}
    A poset $\Po$ is called separative if for all $p, q \in \Po$, if $p \nleq q$, then there exists $r \leq p$ such that for all $s \leq r$, holds $s \nleq q$.
\end{Def}

 Many relevant posets that appear in  forcing are separative. In particular the so called Cohen forcing is separative and, for any complete Boolean algebra $\B$, the poset $\Po = \B \setminus \{0\}$ is separative (just take $r = p . (- q)$).

The (well known) result concerning separative forcing is the following:

\begin{Theo}\label{equivalenceofcategoriesdensemorphismseparative}
	Let $\Po$ be a separative poset. Then there exists a complete Boolean algebra $\mathbb{B}$ such that 
	$$Sh_{\neg \neg}(\Po) \simeq Sh_{\bigvee}(\mathbb{B}).$$
\end{Theo}
\begin{proof}
See \cite{MacMoe94},  Corollary 3 of the Section 4 of the Appendix: Sites for Topoi, and the following comments, pages 590-591.
\end{proof}

One would interpret this theorem as the separative forcing, as a method, has the \emph{same} content as Boolean-valued models. 

Now we will give a generalization of Theorem \ref{equivalenceofcategoriesdensemorphismseparative} for any forcing poset as part of the original contributions of this work (Theorem \ref{HugoMicheltheorem2}).

\begin{Def}
	Let $\C$ be a small category . Let $J$ and $K$ be two Grothendieck topologies on $\C$. We say that $J$ is \textbf{finer} than $K$ if for all $C \in \text{Obj}(\C)$, $J(C)\subseteq K(C)$.
\end{Def}

\begin{Def}
	Let $\C$ and $\D$ be a small categories and $(\D,J)$ a site. Given a functor $F: \C \to \D$, the \textbf{topology induced} by $F$ on $\C$, denoted by $J_{F}$, is the finest one such that for all $H$ sheaf on $\D$ for $J$, $H \circ F$ is a sheaf on $\C$ for $J_{F}$.
\end{Def}

\begin{Lemma}\label{BeilinsonLemma}
	Let $\C$ and $\D$ be  small categories and $J$ a Grothendieck topology on $\D$. Suppose that there exists a faithful functor $F: \C \to \D$ satisfying the following:
	\begin{enumerate}

		\item[] Fix an object $D \in Obj(\D)$. Given a finite set of objects $\{C_{i} \in Obj(\C): i \in I\}$ and a family of morphism in $\D$ of the form $(D \to F(C_{i}))_{i \in I}$, there exists a family of morphisms  in $\D$ of the form $(F(E_{j}) \to D)_{j \in J}$ such that the composition $F(E_{j}) \to D \to F(C_{i})$ lies in the image of $Hom(E_{j},C_{i}) \to Hom_{\D}(F(E_{j}),F(C_{i}))$, for all $i \in I, j \in J$, and the sieve generated by $(F(E_{j}) \to D)_{j \in J}$ covers $D$.
	\end{enumerate}
	Then:
	\begin{enumerate}
		\item The induced topology $J_{F}$ on $\C$ has the following property: For all $C \in Obj(\C)$,
		$$S \in J_{F}(C) \text{ iff } <F(S)> \in J(F(C)),$$
		where $<F(S)>$ is the sieve generated by $ \{F(f): f \in S\}$.
		
		\item $Sh_{J_{F}}(\D) \simeq Sh_{J}(\C)$, i.e., the functor
		\begin{align*}
				F^{*}:\widehat{\D} &\to \widehat{\C} \\
							H &\mapsto H \circ F
		\end{align*}
		restricts to the equivalence of categories above.
	\end{enumerate}
	
\end{Lemma}

\begin{proof}
	See \cite{Bei12} first Proposition of Section 2.1.
\end{proof}

A corollary of Lemma \ref{BeilinsonLemma} is the well-known Comparison lemma:

\begin{Corolary}\label{comparisonlemma}
	Let $\C$ be a small category and $(\D,J)$ a site. Let $F: \C \to D$ be a full and faithful functor. Let $J_{F}$ be the induced topology on $\C$ by $F$. If for all $d \in \text{Obj}(\D)$ there exists a sieve generated by morphisms of the form $(F(C_{i}) \to d)_{i \in I})$, for $C_{i} \in \text{Obj}(\C)$, then
	$$Sh_{J_{F}}(\C) \simeq Sh_{J}(\D).$$
\end{Corolary}

\begin{proof}
	See Lemma \ref{BeilinsonLemma}.
\end{proof}

\begin{Prop}
    \label{sheavesposetalgebrasem0}
	Let $\Po$ be a forcing poset and $\B$ its complete Boolean algebra of regular open sets of $~\Po$. Then
	$$Sh_{\neg \neg}(\Po) \simeq Sh_{\neg \neg}(\B \setminus \{ \emptyset\}).$$
\end{Prop}

\begin{proof}
	Let $i: \Po \to \B$ be the canonical dense morphism. To use Lemma \ref{BeilinsonLemma}, we need to check that the dense morphism $i: \Po \to \B \setminus \{\emptyset\}$ satisfies the required property. So fix $b \in \B \setminus \{\emptyset\}$. Let $$(b \leq  i(a_{j}))_{j \in J}$$ be a finite family of morphism in $\B \setminus \{\emptyset\}$. By density, there exists $a_{b} \in \Po$ such that $i(a_{b})  \leq b$. From $i(a_{b}) \leq b \leq i(a_{j})$, we conclude that $i(a_{b}) \not \perp i(a_{j})$, for all $j \in J$.Then, for each $j \in J$, there exists $r_{j} \in \Po$ such that $r_{j} \leq a_{b}, a_{j}$. Let $S$ be the sieve generated by $(i(r_{j}) \leq b)_{j \in J}$. We need to show that $S \in \neg \neg (b)$. Let $q \leq b$. Fix $j' \in J$. By density, there exists $r_{q} \in \Po$ such that $i(r_{q}) \leq q \cdot i(r_{j'})$. Then $i(r_{q}) \leq q$ and $i(r_{q}) \in S$, because $i(r_{q}) \leq i(r_{j'})$. Note that for all $j \in J$, the composition $i(r_{j}) \leq b \leq i(a_{j})$ lies on the image of 
	$$Hom_{\Po}(r_{j},a_{j}) \to Hom_{\B \setminus \{\emptyset\}}.(i(r_{j}),i(a_{j})).$$
	
	Consider $\neg \neg$ the double-negation topology on $\B \setminus \{\emptyset\}$. By Lemma \ref{BeilinsonLemma}, the induced topology $J_{i}$ provides an equivalence of categories of sheaves
	
	$$Sh_{J_{i}}(\Po) \simeq Sh_{\neg \neg}(\B \setminus \{ \emptyset \}).$$
	
	We now show that $J_{e}$ is the double negation topology $\neg \neg$. Fix a $p \in \Po$ and let $$S = (p_{j} \leq p)_{j \in J} \in \neg \neg(p).$$ We will show that  $<i(S)> \in \neg \neg (i(p))$. Let $b \in \B \setminus \{\emptyset\}$ such that $b \leq i(p)$. By density, there exists $a \in \Po$ such that $i(a) \leq b$. In particular, $i(a)\not \perp i(p)$, so there exists $c \in \Po$ such that $c \leq a,p	$. But $S$ is a sieve, then $ (c \leq p) \in S$. We conclude that $(i(c)\leq i(p)) \in <i(S)>$ and $i(c) \leq b.$
	
	Conversely, suppose that $<i(S)> \in \neg \neg (i(p))$. Fix a $q \leq p$. In particular, $i(q) \leq i(p)$. Then, there exists $(s \leq i(p)) \in <i(S)>$ so that $s \leq i(q)$. Moreover, there exists $u \in S$ such that $i(u) \leq s$. From $i(u) \not \perp i(q)$, we conclude that there is $v \in \Po$ such that $v \leq u,q$. Then $(v\leq i(p)) \in S$ and $v\leq q$.
\end{proof}

\begin{Prop}\label{sheavesnegationandsup}
	Let $\B$ be a complete Boolean algebra. Then
	$$Sh_{\neg \neg}(\B \setminus{\{0\}}) \simeq Sh_{\bigvee}(\B),$$
	where $\bigvee$ is the sup (Grothendieck) topology on $\B$.
\end{Prop}
\begin{proof}
	Denote by $e: \B \setminus \{ 0 \} \to \B$ the inclusion function. Note that since $e$ is the identity map, the property required to use Corollary \ref{comparisonlemma} is satisfied. Consider the sup topology $\bigvee$ on $\B$. By Corollary \ref{comparisonlemma}, the induced topology $J_{e}$ on $\B \setminus \{ 0\}$ provides an equivalence of categories of sheaves 
	$$Sh_{J_{e}}(\B \setminus \{0\}) \simeq Sh_{\bigvee}(\B).$$
	It remains to proof that $J_{e} = \neg \neg$. Fix $b \in \B \setminus \{ 0\}$ and let $S$ be a sieve that covers $b$. We will show that
	$$S \in \neg \neg(b) \text{ iff } <e(S)> \in \bigvee (e(b)).$$
	
	Suppose that $S \in \neg \neg (b)$. Denote $S$ by $ \{b_{j} \leq b\}_{j \in J}$. Then 
	$$<e(S)> = \{b_{j} \leq b\}_{j \in J} \cup \{0 \leq b\}. $$
	
	We want to show that 
	$$\bigvee \{b_{j} \leq b\}_{j \in J}\cup \{ 0 \leq b\} = \bigvee \{b_{j}\}_{j \in J} = b.$$
	
	Note that for all $j \in J$, $b_{j} \leq b$.  Therefore $\bigvee \{b_{j}\}_{j \in J} \leq b$. Suppose that it is not true that $b \leq \bigvee \{b_{j}\}_{j \in J}$. Denote by $x = \bigvee \{b_{j}\}_{j \in J}.$ Define $p = b \cdot (-x).$ Note that $p \leq b$ and $p \neq 0$. Then, there exists $j' \in J$ such that $b_{j'} \leq p$. In particular,
	$$b_{j'} \leq p \cdot (-x) \leq p \cdot (-b_{j'}).$$
	Then $b_{j'} = b_{j'}\cdot(p\cdot(-b_{j'})) = 0$. Contradiction. 
	
	On the other hand, suppose that $\bigvee\{b_{j}\}_{j \in J} = b.$ Fix $q \leq b$. Then
	$$q = q \cdot b = q \cdot(\underset{j \in J}{\bigvee}b_{j}) = \underset{j \in J}{\bigvee} q \cdot b_{j}.$$
	
	Then, there exists $j_{1} \in J$ such that $q \cdot b_{j_{2}} = b_{j_{1}}$, for some $j_{2} \in J$ (otherwise $q = 0$, which is not the case). Then $b_{j_{1}} \leq q$.
\end{proof}

\begin{Theo}\label{HugoMicheltheorem2}
	Let $\Po$ be a forcing poset. Then there exist a complete Boolean algebra $\B$ such that
	$$Sh_{\neg \neg}(\Po) \simeq Sh_{\bigvee}( \mathbb B).$$
\end{Theo}
\begin{proof}
	The result follows from Proposition \ref{sheavesposetalgebrasem0} and Proposition \ref{sheavesnegationandsup}.
\end{proof}

\section{Final remarks and future work}\label{futureworks}

\subsection{Categorical relationships }

We point out here that the connections presented in Section \ref{forcingmodelsandsheaves} can be expanded from the classical/Boolean framework to the 
intuitionistic/Heyting  setting.

Let $\He$ be a complete Heyting algebra (a.k.a., locale). It is possible to define a \emph{Heyting-valued} model, $V^{\He}$, using the same procedure used to construct a Boolean-valued model, $V^\B$, but replacing the complete Boolean algebra $\B$ by the complete Heyting algebra $\He$.

From  $V^\He$  it can be defined a category $Set^{(\He)}$  "by taking quotients" and this can be shown a category equivalent with the category  of $\He$-sets endowed with {\em relational} morphisms (see \cite{Alvim22} for a detailed account).

$$V^\He \mapsto Set^{(\He)} \simeq \mathbb{H}-sets_{rel}$$

Moreover, in Chapter 2 in \cite{Borceux94} we can find equivalences of categories\footnote{In fact, there are three intermediary equivalences: $\mathbb{H}-sets_{rel} \simeq \mathbb{H}-sets_{rel}^{compl}  \simeq \mathbb{H}-sets_{func}^{compl} \simeq Sh_{\bigvee(\mathbb{H})}$, where {\em compl} refers to complete $\mathbb{H}$-sets and {\em func} refers to functional morphisms.}:

 $$\mathbb{H}-sets_{rel} \simeq Sh_{\bigvee(\mathbb{H})}$$

A result analogous to the Theorem \ref{HugoMicheltheorem2} can be obtained:

    $$Sh(\Po) \simeq Sh_{\bigvee}(\mathbb H),$$
    
    where $\mathbb H$ is the complete Heyting algebra of all open subsets of the topological space $\mathbb P$.

Therefore, in the same vein as in Section \ref{forcingmodelsandsheaves}, we have a natural connection between intuitionistic forcing (see \cite{Mariano05}), Heyting-valued models and localic toposes (i.e. Grothendieck toposes that are equivalent to $Sh_{\bigvee}(\mathbb H)$, for some complete Heyting algebra $\He$).

\vspace{0.5cm}

On the other hand, for every Grothendieck topos $\cal G$, we have a (essentialy unique) {\em simulation} of von Neumann hierarchy inside $\cal G$ (see \cite{Fourman80}, \cite{Hayashi81}, \cite{Streicher09}):

\begin{itemize}
    \item  $V_0({\cal G}) = 0$; 
 
\item $V_{\alpha+1}({\cal G})  = P(V_{\alpha}({\cal G}))$; 
 
\item $V_\lambda({\cal G}) =  colim_{\beta<\lambda} V_\beta({\cal G})$, if $\lambda$ is a limit ordinal;

\end{itemize}
where $0$ is "the" initial object of $\cal G$ and $P(X)$ is "the" power of the object $X$ in $\cal G$

Then we can simulate a von Neumann universal class inside $\cal G$.

${\cal G} \mapsto V(\cal G) :=$ "$colim_{\alpha \in ON}$" $V_\alpha(\cal G)$

It can be interesting determine the result of application of both processes and provide comparisons from the original data with the ones obtained from iterated processes.

$$V^\He \mapsto {\cal G}(V^\He):= Set^{(\He)}  \mapsto V({\cal G}(V^\He))$$

$${\cal G} \mapsto V(\cal G)^{\He(\cal G)} \mapsto {\mathcal G} {Set}^{(\He(\cal G))}$$

where ${\He(\cal G)}$ is the locale of subobjects of the terminal object $1$ in $\cal G$, and ${\cal G}Set^{(\He(\cal G))}$ denotes the category obtained from the Heyting-valued model $V(\cal G)^{\He(\cal G)}$ (several possibilities here) by taking quotients. 


Moreover, it should be interesting analyze the behavior of the above comparison when taking into account  "changes of bases" given by convenient morphisms $\He \to \He'$ and ${\cal G} \to {\cal G}'$.

\subsection{Semantical relationships}


There is a hindrance to directly comparing the natural semantics in a model of set theory and in a topos: while the basic language of sets is untyped, the more widespread and perhaps more natural natural language of topos, the Mitchell-Bénabou language, is typed by the topos objects. In particular, in the set-theoretical context, quantifications would naturally be unbounded (since all variables are of the same type), and in the topos context, quantifications would be bounded (since all variables are of specific types). As is known and already mentioned, both semantics can be seen as embodiments of the concept of forcing (for toposes, this is called Kripke-Joyal semantics).


While this creates some difficulties in making comparisons at first, it also generates a range of opportunities to establish comparisons. Below, we outline some developments already made and other possible strategies for future investigations.





In \cite{Fourman80}, \cite{Hayashi81}, and \cite{Streicher09}, the strategy adopted was basically to use the simulation of a von Neumann hierarchy internal to a Grothendieck topos ${\cal G}$ to define unbounded quantifications in this topos by:

$$[[\exists x.\phi]] = \bigvee_{\alpha \in On} [[\exists x: V_\alpha({\cal G}).\phi]]$$

$$[[\forall x.\phi]] = \bigwedge_{\alpha \in On} [[\forall x: V_\alpha({\cal G}).\phi]]$$

where the symbol $x : b$ indicates that the variable x has type $b$ ($b$ is an object of the topos $\cal G$) and the symbol $[[\psi]]$ indicates that, for a formula $\psi$ with free variables $x_1:b_1, \ldots, x_n:b_n$, $[[\psi]]$ is a certain morphism in $\cal G$, $[[\psi]] \in Hom_{\cal G}(b_1 \times \ldots \times b_n,\Omega)$ (which in turn is a complete Heyting algebra), and $\Omega$ is "the" subobject classifier of $\cal G$ (in particular, if $\psi$ is a sentence, then $[[\psi]$ can be identified with a subobject of $1$, the terminal object of the topos $\cal G$).


Despite the restriction involved (since the Von Neumann universe in a topos considers only a fragment of it), this approach, not yet widely explored, shows promise, because in particular Theorem 4.1 in \cite{Fourman80}, states (without explicitly proving it) that for a complete Boolean algebra $\B$, the semantics of $V^\B$ and the semantics of $V(Set^{(\B)})$ (as in the previous subsection) coincide. We recall that $V^\B$ is a model of $ZFC$ and that $V^\He$ is a model of $IZF$, an intuitionistic counterpart (therefore without the axiom of choice) of $ZF$.


Another approach that essentially goes in the same direction, but expands in the previous approach, is M. Shulman's stack semantics for topos a ${\cal G}$ (in fact, for a Heyting category satisfying a technical condition), see \cite{Shulman10}. Therein, one can simulate unbounded quantifications in ${\cal G}$, roughly through a combination of processes: one can quantify both on the variables typed by objects of the slice topos ${\cal G}{\downarrow} b$ (as already occurs in the Kripke-Joyal semantics) and on all objects $b$ of the topos $\cal G$, and not only on the variables typed by objects of the special form $V_\alpha(\cal G)$ as in the approach of \cite{Fourman80, Hayashi81}.


Finally, we can also try to compare the formulas with restricted quantifications --i.e., the formulas of the usual language of topos and the formulas with bounded quantifiers\footnote{$\exists x \in v . \psi(x) := \exists x (x \in v \wedge \psi(x))$; $\forall x \in v . \psi(x) := \forall x (x \in v \to \psi(x))$.} of $L_{ZF}$-- and their corresponding semantics, considering the connections between models with values in Heyting algebras and Grothendieck's topos outlined at the end of the previous subsection. In more detail:

\begin{itemize}



\item On the one hand, from the  $IZF$  model $V^\He$ to the topos $Set^{(\He)} \simeq Sh(\He)$, we can compare those restricted formulas of $L_{ZF}$, i.e., only formulas of $L_{ZF}$ where the quantifications are bounded (in this language we eventually join constants for each name of $V^\He$), and their Heyting-value (= maximum of the values that force this formula) with the value given by the supremum (= maximum) of the values obtained by the Kripke-Joyal semantic of the corresponding formula in the Mitchell-Bénabou language of the localic topos $Set^{(\He)}$.

\item On the other hand, from the topos $\cal G$ to the universe $V(\cal G)$, compare the semantic values of the typed formulas with typed terms generated by $\cal G$ with the semantic values of the corresponding formula with bounded quantifications from the corresponding set-theoretical universe $V(\cal G)$.

\end{itemize}


\end{document}